\documentclass[12pt]{amsart}
\usepackage[lite]{amsrefs} 
\usepackage{amssymb,amsmath,enumitem}
\usepackage{graphicx}
\newcommand{\R}{\mathbb{R}}

\newcommand{\W}{{\bf{W}}}
\newcommand{\K}{{\bf{K}}}
\newcommand{\I}{{\bf{I}}}

\numberwithin{equation}{section}
\usepackage{hyperref}
\hypersetup{colorlinks=true, citecolor=blue, linkcolor=red, urlcolor=blue}
\theoremstyle{plain}
\newtheorem{Thm}{Theorem}[section]
\newtheorem{Cor}[Thm]{Corollary}
\newtheorem{Lem}[Thm]{Lemma}

\theoremstyle{definition}

\newtheorem{Rem}[Thm]{Remark}

\theoremstyle{remark}

\begin{document}
\title[Bilateral estimates of solutions]
{Bilateral estimates of solutions\\ to  quasilinear elliptic equations\\ with sub-natural growth terms}

\author{I. E. Verbitsky}
\address{Department of Mathematics, University of Missouri,  Columbia,   \newline 
Missouri 65211, USA}
\email{\href{mailto:verbitskyi@missouri.edu}{verbitskyi@missouri.edu}}
\begin{abstract} We study 
 quasilinear  elliptic equations  of the type 
$
-\Delta_{p} u  = \sigma u^{q}  + \mu  \; \; \text{in} \;\; \R^n  
$ 
in the 
case $0<q< p-1$, where $\mu$ and  $\sigma$ are nonnegative measurable functions, or  locally finite measures,  and  $\Delta_{p}u= \text{div}(|\nabla u|^{p-2}\nabla u)$ is the $p$-Laplacian. Similar equations with 
more general local and nonlocal operators in place of $\Delta_{p}$ 
are treated as well.  

We obtain existence criteria and global bilateral pointwise estimates for all positive solutions $u$:  
$$ u(x) \approx   (\W_p \sigma(x))^{\frac{p-q}{p-q-1}} + \K_{p,q} \sigma(x) + \W_p \mu (x), \quad x \in \R^n,$$
where $\W_p$ and $ \K_{p, q} $ are, respectively, the Wolff potential and the intrinsic Wolff potential, with the constants of equivalence depending only on   $p$, $q$ and $n$. 

The contributions of $\mu$ and $\sigma$ in these pointwise estimates 
are totally separated,  which 
is a new phenomenon even when $p=2$.  In the homogeneous case $\mu=0$,  such estimates  
 were obtained earlier by a different method  only for  minimal positive solutions. 
  
\end{abstract}
\subjclass[2010]{Primary 35J92, 42B37; Secondary 35J20.} 
\keywords{Wolff potentials, $p$-Laplacian, pointwise estimates}

\maketitle
\tableofcontents
\section{Introduction}\label{sect:intro}

We present  a new approach to pointwise estimates of solutions 
 to   quasilinear elliptic equations of the type 
 \begin{equation}\label{inhom}
\begin{cases}
-\Delta_{p} u  = \sigma u^{q} + \mu, \quad u\ge 0  \quad \text{in} \;\; \R^n, \\
\liminf\limits_{x\rightarrow \infty}u(x) = 0,  
\end{cases}
\end{equation}
where $\mu$, $\sigma\ge 0$ are locally 
integrable functions, or Radon measures (locally finite)  in $\R^n$, in the 
\textit{sub-natural growth} case $0<q<p-1$. 

In this paper, all solutions $u$ (possibly unbounded)  are understood to be  $p$-superharmonic (or equivalently locally renormalized)  solutions   (see \cite{KKT}). We will  assume that  
 $u\in L^{q}_{{\rm loc}}(\R^n, d\sigma)$, so that the right-hand side of \eqref{inhom} is a Radon measure.

We will obtain matching upper and lower estimates of solutions in terms of nonlinear potentials 
defined below. Our estimates hold for all $p$-superharmonic solutions $u$. In particular, they yield an existence criterion for solutions to  \eqref{inhom}.

In the special case $\mu=0$, i.e., 
\begin{equation}\label{q-hom}
\begin{cases}
-\Delta_{p} u  = \sigma u^{q}, \quad u\ge 0  \quad \text{in} \;\; \R^n, \\
\liminf\limits_{x\rightarrow \infty}u(x) = 0,  
\end{cases}
\end{equation}
considered earlier in \cite{CV1}, the upper pointwise estimate was obtained only for the 
minimal solution $u$. Our proofs are new even in this case. 

When $p=2$ and $0<q<1$, these sublinear elliptic equations were studied by Brezis and Kamin   \cite{BK} (see also \cite{CV2}, \cite{SV}, \cite{QV}, \cite{V}, 
and the literature cited there).

The case $q\ge p-1$,  
which comprises Schr\"{o}dinger type equations with \textit{natural growth} terms 
when $q=p-1$, and  equations of superlinear type when $q>p-1$, is quite different  in nature 
(see, for instance,   
\cite{JMV}, \cite{JV}, \cite{PV1}, \cite{PV2}).

We observe that in general, for the existence of a nontrivial solution $u$ to 
\eqref{inhom},  $\sigma$ must be absolutely continuous with respect to $p$-capacity, i.e., 
$\sigma(K)=0$ whenever $\text{cap}_p (K)=0$, for any compact set $K$ in $\R^n$. Here the $p$-capacity of $K$ is defined by 
\begin{equation}\label{p-capacity}
\text{cap}_p (K)=\inf \left\{\int_{\R^n} |\nabla u|^p dx: \, \,  u \ge 1 \, \, \text{on} \, \, K, \quad u \in C^\infty_0(\R^n)\right \}.
\end{equation}

More precisely, if $u$ is a nontrivial (super)solution  to \eqref{q-hom} in the case $0<q\le p-1$, then  (see \cite{CV1}*{Lemma 3.6} for a more general estimate)  
\begin{equation}\label{abs-cap0}
\sigma(K)\le \text{cap}_p(K)^{\frac{q}{p-1}}\left(\int_K u^q d\sigma\right)^{\frac{p-1-q}{p-1}},
\end{equation}
for all compact sets $K\subset \R^n$.

 Among our main tools are certain \textit{nonlinear potentials} associated with \eqref{q-hom}. We refer to the recent survey of nonlinear potentials and their applications to PDE by Kuusi and Mingione \cite{KuMi}. 
 
 Let $ M^{+}(\R^n)$ denote the class of all positive (locally finite) Radon measures 
on $\R^n$. 
Given  a measure 
$\sigma \in M^{+}(\R^n)$, 
 $1<p<\infty$ and 
$0<\alpha<\frac{n}{p}$, 
the Havin-Maz'ya-Wolff potential, introduced in \cite{HM} (see also \cite{HeWo}),   is defined by 
\begin{equation}\label{wolff-pq}
\W_{\alpha, p}\sigma (x) = \int_{0}^{\infty} \left[ \frac{\sigma(B(x,t))}{t^{n-\alpha p}}\right]^{\frac{1}{p-1}}\; \frac{dt}{t}, \quad x \in \R^n,
\end{equation}
where $B(x,t)$ is a ball of radius $t >0$ centered at $x \in \R^n$. 

Nonlinear potentials $\W_{\alpha, p}\sigma$,  
often called 
Wolff potentials, occur in various problems of harmonic analysis, approximation theory, Sobolev spaces,  in particular spectral synthesis problems (\cite{AH},  \cite{HM},  \cite{HeWo}, \cite{Maz}), as well as  quasilinear (\cite{KiMa}, \cite{MZ}, \cite{PV1}) and fully nonlinear PDE  (\cite{Lab}, \cite{TW1}, \cite{TW2}). 

In the linear case $p=2$, clearly $\W_{\alpha, p}\sigma = \I_{2\alpha}\sigma$  (up to a constant multiple), where the Riesz potential of order $\beta\in (0, n)$ is defined by 
\[
\I_{\beta}\sigma(x)=\int_{\R^n} \frac{d \sigma (y)}{|x-y|^{n-\beta}}, \quad x \in \R^n.  
\]

In the special case $\alpha=1$, we will be using the notation $\W_p \sigma=\W_{1,p} \sigma$ ($1<p<n$), i.e., 
\begin{equation}\label{wolff-p}
\W_p \sigma (x) = \int_{0}^{\infty} \left[ \frac{\sigma(B(x,t))}{t^{n-p}}\right]^{\frac{1}{p-1}}\; \frac{dt}{t}, \quad x \in \R^n.
\end{equation}
 These potentials are  intimately related to 
 the equation 
\begin{equation}\label{p-lap}
\begin{cases}
-\Delta_{p} u = \sigma, \quad u \ge 0 \quad \text{in} \;\; \mathbb{R}^n, \\ 
\liminf\limits_{x \rightarrow \infty} u(x) = 0, 
\end{cases}
\end{equation}
where $\sigma \in M^{+}(\R^n)$.

The following important global estimate, along with its local counterpart,  is due to T.~Kilpel\"{a}inen and J.~Mal\'y 
\cite{KiMa}:  \textit{Suppose 
$u\ge 0$ is a $p$-superharmonic solution to \eqref{p-lap}. Then
\begin{equation}\label{k-m}
K^{-1}\W_p \sigma (x) \leq u(x) \leq K \W_p \sigma(x),
\end{equation} 
where $K=K(n,p)$ is a positive constant.}

It is known that  a nontrivial solution $u$  to \eqref{p-lap} 
exists if and only if 
\begin{equation}\label{k-m-cond}
\int_{1}^{\infty}  \left[ \frac{\sigma(B(0,t))}{t^{n-p}} \right]^{\frac{1}{p-1}}\frac{d t}{t}<\infty.
\end{equation} 
This is equivalent to $\W_p \sigma(x)<\infty$ for some  $x \in \R^n$, or 
equivalently quasi-everywhere (q.e.) on $\R^n$. 
In particular, \eqref{k-m-cond} may hold only in the case $1<p<n$, unless $\sigma= 0$.


The following bilateral pointwise estimates of nontrivial (minimal) solutions $u$ to 
\eqref{q-hom} in the  case $0<q<p-1$  are fundamental 
to our approach   (see \cite{CV1}, where the upper estimate was proved only for 
the minimal solution):
\begin{equation} \label{two-sided}
\begin{aligned}
c^{-1} & \left [(\W_p \sigma(x))^{\frac{p-1}{p-1-q}}+\K_{p, q}  \sigma(x)\right] 
\le u (x) \\ &\le c \left[(\W_p \sigma(x))^{\frac{p-1}{p-1-q}} +\K_{p, q}  \sigma(x) 
\right],  \quad x\in \R^n, 
\end{aligned}
\end{equation}
where $c>0$ is a constant which  depends only on $p$, $q$, and $n$. 

Here $\K_{p, q} \sigma$ is the so-called \textit{intrinsic} nonlinear potential 
associated with \eqref{q-hom}, which was introduced in \cite{CV1}.  
It is defined in terms of the  localized weighted norm inequalities,    
\begin{equation} \label{weight-lap}
\left(\int_{B} |\varphi|^q \, d \sigma\right)^{\frac 1 q} \le \varkappa(B) \,  ||\Delta_p \varphi||^{\frac{1}{p-1}}_{L^1(\R^n)},  
\end{equation} 
for all test functions $\varphi$ such that  $-\Delta_p \varphi \ge 0$, $\displaystyle{\liminf_{x \to \infty}} \, \varphi(x)=0$. Here $\varkappa(B)$ denotes the least constant in \eqref{weight-lap} 
associated with the measure $\sigma_B=\sigma|_B$ restricted to a ball 
$B=B(x, t)$.  
Then the \textit{intrinsic} nonlinear potential $\K_{p, q}\sigma$ is defined by 
\begin{equation} \label{potentialK}
\K_{p, q}  \sigma (x)  =  \int_0^{\infty} \left[\frac{ \varkappa(B(x, t))^{\frac{q(p-1)}{p-1-q}}}{t^{n- p}}\right]^{\frac{1}{p-1}}\frac{dt}{t}, \quad x \in \R^n.
\end{equation} 
As was noticed in \cite{CV1},  $\K_{p, q}  \sigma \not\equiv + \infty$ if and only if 
\begin{equation}\label{suffcond1}
\int_1^{\infty} \left[\frac{ \varkappa(B(0, t))^{\frac{q(p-1)}{p-1-q}}}{t^{n- p}}\right]^{\frac{1}{p-1}}\frac{dt}{t} < \infty.  
\end{equation}

Consequently, a nontrivial $p$-superharmonic solution $u$ to  \eqref{q-hom} exists if and only if both $\K_{p, q}  \sigma \not\equiv + \infty$ and 
$\W_p \sigma\not\equiv + \infty$, that is, both \eqref{k-m-cond} 
and \eqref{suffcond1} hold.

For the existence of a nontrivial solution to equation \eqref{inhom}, we need to add 
the condition  $\W_p \mu\not\equiv + \infty$, i.e.,
\begin{equation}\label{suff-cond}
\int_{1}^{\infty}  \left[\frac{\mu(B(0,t))}{t^{n-p}}\right]^{\frac{1}{p-1}} \frac{dt}{t} <\infty.  
\end{equation}

In this paper, we obtain the following the following criterion for existence, along with global bilateral estimates of solutions to \eqref{inhom}. 

\begin{Thm}\label{thm:main1}
Let $1<p<n$, $0<q<p-1$, and $\mu, \sigma \in M^{+}(\R^n)$. There exists a 
nontrivial solution $u$ to \eqref{inhom} if and only if conditions \eqref{k-m-cond}, 
\eqref{suffcond1}, and \eqref{suff-cond} hold. Then any nontrivial solution $u$  satisfies the estimates 
\begin{equation}\label{main-b1} 
\begin{aligned}
 C_1 \, & \left[  (\W_p \sigma(x))^{\frac{p-1}{p-1-q}} + \K_{p, q} \sigma(x)  + \W_p \mu (x)\right] \le u(x) \\ 
\le & C_2 \, \left[  (\W_p \sigma(x))^{\frac{p-1}{p-1-q}} + \K_{p, q} \sigma(x) + \W_p \mu (x)\right], \quad x \in \R^n. 
\end{aligned}
\end{equation}
where the positive constants $C_1, C_2$  depend only  on $p, q$, and $n$. 

If $n \leq p < \infty$, then there are no nontrivial solutions to \eqref{inhom}.  
\end{Thm}


The following corollary is deduced  from Theorem \ref{thm:main1}  under the additional assumption that there exists a constant $C=C(\sigma, p, n)$ so that 
\begin{equation}\label{cap-p}
\sigma(K)\le C\, \text{cap}_p(K), \quad \textrm{for all compact sets} \,  \, 
K \subset \mathbb{R}^n.
\end{equation}
We remark that condition  \eqref{cap-p} is also essential in the natural growth case 
$q=p-1$ (see, for instance, \cite{JMV}).

\begin{Cor}\label{cor1} Let $1<p<n$, $0<q<p-1$, and $\mu, \sigma \in M^{+}(\R^n)$. 
If condition \eqref{cap-p} holds, then 
 any positive solution $u$ to \eqref{inhom} satisfies the estimates 
\begin{equation}\label{main-bb} 
\begin{aligned}
 C_1 \, & \left[  (\W_p \sigma(x))^{\frac{p-1}{p-1-q}}  + \W_p \mu (x)\right] \le u(x) \\ 
\le & C_2 \, \left[  (\W_p \sigma(x))^{\frac{p-1}{p-1-q}} + \W_p  \sigma(x) + \W_p \mu (x)\right], \quad x \in \R^n. 
\end{aligned}
\end{equation}
where  $C_1, C_2$  are positive constants that depend only  on $p, q, n$, and  the constant $C$ in \eqref{cap-p} (in the case of $C_2$). 
\end{Cor}

In the special case $\mu=0$,  the Brezis--Kamin type pointwise estimates 
\begin{equation} \label{b-k}
 C_1 \, (\W_p \sigma(x))^{\frac{p-1}{p-1-q}} \le u(x) \le C_2 \, [(\W_p  \sigma(x))^{\frac{p-1}{p-1-q}} + \W \sigma(x)],  
\end{equation}
under the assumption  \eqref{cap-p} were obtained in \cite{CV2} (the upper estimate was proved only for the minimal solution). For 
\textit{bounded} solutions $u$,  the term $(\W_p \sigma)^{\frac{p-1}{p-1-q}}$ on the right-hand side  of \eqref{b-k}  is redundant. 
This estimate in the case $p=2$ was originally obtained in \cite{BK}.

Our main results are deduced via pointwise estimates of solutions to the fractional integral equation 
\begin{equation} \label{q-frac}
u = \W_{\alpha, p} (u^{q} \sigma) +  \W_{\alpha, p} \mu, \quad u\ge 0, \quad  \text{in} \;\; \mathbb{R}^n,
\end{equation}
where $0<q<p-1$, $0<\alpha<\frac{n}{p}$, and $\mu, \sigma \in M^{+}({R}^n)$.

Bilateral pointwise estimates of solutions to \eqref{q-frac}, similar to \eqref{two-sided}, are given in terms of nonlinear potentials $\W_{\alpha, p}$ 
and  fractional \textit{intrinsic} potentials $\mathbf{K}_{\alpha, p, q}$ defined in  Sec. \ref{sect:wolff}. 
In the definition of $\mathbf{K}_{\alpha, p, q}$, which is similar to  \eqref{potentialK} in the case $\alpha=1$, we employ localized embedding constants $\kappa(B)$ associated 
with $\sigma_B$ for balls $B=B(x, r)$, which are related to  certain weighted norm inequalities for potentials 
$\W_{\alpha, p}$. 

In the special case $p=2$, $0<q<1$, $0<2\alpha<n$, we obtain an analogue of Theorem  \ref{thm:main1}
  for the fractional Laplace problem
\begin{equation} \label{eq:frac-laplacian}
\begin{cases}
\left(-\Delta \right)^{\alpha} u = \sigma u^{q} +\mu, \quad u \ge 0  \quad \text{in} \;\; \mathbb{R}^n, \\
\liminf\limits_{x \rightarrow \infty}u(x) = 0.
\end{cases}
\end{equation}

 Our results on solutions to \eqref{q-frac} 
  demonstrate  (see Sec. \ref{sec4} below)  that Theorem  \ref{thm:main1} remains valid for more general quasilinear operators $\text{div}\mathcal{A}(x, \nabla u)$ in place of $\Delta_p$,   
  under standard boundedness and monotonicity assumptions on $\mathcal{A}(x, \xi)$ 
  (with $\alpha=1$, $0<q<p-1$),  as well as for $k$-Hessian operators (with 
$\alpha=\frac{2k}{k+1}$, $p=k+1$ and $0<q<k$). The relation 
between equations  \eqref{q-frac} and the corresponding elliptic PDE
is provided by the 
nonlinear potential theory developed in 
\cite{KuMi}, \cite{Lab},  \cite{TW2}.

If $q\ge p-1$ for the quasilinear 
equations, or $q\ge k$ for the $k$-Hessian equations, the existence results and pointwise estimates of solutions  differ greatly from Theorem  \ref{thm:main1}. They 
were  obtained  
earlier in \cite{JV}, \cite{PV1}, \cite{PV2}. 

This paper is organized as follows. 
In Sec. \ref{sect:wolff}, we recall definitions of the nonlinear potentials $\mathbf{W}_{\alpha, p}$ and 
$\mathbf{K}_{\alpha, p, q}$, and discuss some of their properties. Pointwise estimates 
of sub- and super-solutions of  
the homogeneous equation \eqref{q-hom} are discussed in Sec. \ref{sec3}. They are extended to the non-homogeneous equation \eqref{inhom} 
 in Sec. \ref{sec4}, where we prove Theorem  \ref{thm:main1},  and its  analogues for equation \eqref{q-frac}.


\section{Nonlinear potentials}\label{sect:wolff}

Let $1<p<\infty$, 
$0<\alpha< \frac{n}{p}$, and $0<q<p-1$. Let $\sigma \in M^{+}(\R^n)$. 
For the sake of simplicity, the nonlinear potential $\W_{\alpha, p}\sigma$ defined in the Introduction  
will be denoted by  $\W\sigma$, i.e., 
\begin{equation}\label{wolff}
\W\sigma (x) = \int_{0}^{\infty} \left[ \frac{\sigma(B(x,t))}{t^{n-\alpha p}}\right]^{\frac{1}{p-1}}\; \frac{dt}{t}, \quad x \in \R^n. 
\end{equation}

We denote by $\kappa$ the least constant in 
the weighted norm inequality 
\begin{equation} \label{kap-global}
||\W \nu||_{L^q(\R^n, d\sigma)} \le \kappa  \, \nu(\R^n)^{\frac{1}{p-1}}, \quad \forall \nu \in M^{+}(\R^n).  
\end{equation}
We will also need a localized version of \eqref{kap-global} for $\sigma_E=\sigma|_E$, where $E$ is 
a Borel subset   of $\R^n$, and $\kappa(E)$ is the least constant in 
\begin{equation} \label{kap-local}
||\W  \nu||_{L^q(d\sigma_{E})} \le \kappa (E) \, \nu(\R^n)^{\frac{1}{p-1}}, \quad \forall \nu \in M^{+}(\R^n). 
\end{equation}
In applications, it will be enough to use $\kappa(E)$ 
where $E=Q$ is a dyadic cube, or $E=B$ is a ball in $\R^n$. 

It is easy to see using estimates \eqref{k-m} that embedding constants 
$\kappa(B)$ in the case $\alpha=1$ are equivalent to the constants 
$\varkappa(B)$ in \eqref{weight-lap}.

We define the \textit{intrinsic} potential of Wolff type $\K \sigma= \K_{\alpha, p, q} \sigma$ 
in terms of $\kappa(B(x, t))$, the least constant in \eqref{kap-local} with $E=B(x, t)$: 
\begin{equation} \label{intrinsic-K}
\K \sigma (x)  =  \int_0^{\infty} \left[\frac{ \kappa(B(x, t))^{\frac{q(p-1)}{p-1-q}}}{t^{n- \alpha p}}\right]^{\frac{1}{p-1}}\frac{dt}{t}, \quad x \in \R^n.  
\end{equation} 

Notice that $\K_{\alpha, p, q} \sigma (x)\approx  \K_{p, q} \sigma(x)$ 
 in the case $\alpha=1$, with the equivalence constants that depend only on $p$, $q$, and $n$ (see 
\cite{CV1}). It is easy to see that  $\K \sigma (x)  \not\equiv \infty$ if and only if 
\begin{equation} \label{inf}
 \int_a^{\infty} \left[\frac{ \kappa(B(0, t))^{\frac{q(p-1)}{p-1-q}}}{t^{n- \alpha p}}\right]^{\frac{1}{p-1}}\frac{dt}{t}< \infty,
\end{equation} 
for any (equivalently, all) $a>0$. This is similar to the condition $\mathbf{W}_{\alpha, p} \sigma (x)  \not\equiv \infty$, 
 equivalent to (see, for instance,  \cite{CV1}*{Corollary 3.2}) 
\begin{equation} \label{inf-w}
 \int_a^{\infty} \left[\frac{ \sigma(B(0, t))}{t^{n- \alpha p}}\right]^{\frac{1}{p-1}}\frac{dt}{t}< \infty.
\end{equation}


\section{Homogeneous equations}\label{sec3}

Let $1<p<\infty$, 
$0<\alpha< \frac{n}{p}$, and $0<q<p-1$. Let us fix $\sigma \in M^{+}(\R^n)$.  We start with some estimates of solutions to the equation 
\begin{equation}\label{q-eq}
u(x) = \W (u^q d\sigma)(x), \quad u \ge 0,   \quad x \in  \R^n, 
\end{equation}
where $u<\infty$ $d \sigma$-a.e. (or equivalently $u \in L^q_{{\rm loc}}(\R^n, \sigma)$). 
Equation \eqref{q-eq} can also we considered pointwise at every $x \in \R^n$ where
 $u(x) = \W (u^q d\sigma)(x)<\infty$. 

We also treat the corresponding 
subsolutions $u \ge 0$ such that 
\begin{equation}\label{q-sub}
u(x)\le \W (u^q d \sigma)(x)<\infty, \quad x \in  \R^n, 
\end{equation}
and supersolutions $u \ge 0$ such that  
\begin{equation}\label{q-sup}
\W (u^q d \sigma(x)\le u(x) <\infty, \quad x \in  \R^n, 
 \end{equation}
 considered either $d \sigma$-a.e.,  or 
 at  every $x \in \R^n$ where these inequalities hold. 

For any $\nu \in M^{+}(\R^n)$ ($\nu \not=0$) such that $\W \nu \not\equiv \infty$, we set 
\begin{equation}\label{phi}
\phi_\nu(x) := \W \nu(x) \left( \frac{\W[(\W\nu)^q d \sigma](x)}{\W \nu(x)}\right)^{\frac{p-1}{p-1-q}}, \quad x \in \R^n,
\end{equation}
where we assume that $\W \nu(x)<\infty$.

Next, for $x \in \R^n$, we set 
\begin{equation}\label{phi-def}
 \phi (x) := \sup \{ \phi_\nu(x): \,\, \nu \in  M^{+}(\R^n),  \, \nu\not=0, \, \W \nu(x) < \infty\}.
\end{equation}

\begin{Thm}\label{thm:main2}
Let $1<p<\infty$, 
$0<\alpha< \frac{n}{p}$, and $0<q<p-1$. Let $\sigma \in M^{+}(\R^n)$. Then any nontrivial solution $u\ge 0$ to 
\eqref{q-eq} satisfies the estimates 
\begin{equation}\label{main-b} 
 C \,  \phi(x)  \le u(x) \le   \phi(x) , \qquad x \in \R^n,  
\end{equation}
where  $C$  is a positive constant which depends only  on $p$, $q$, $\alpha$ and $n$. 

Moreover, the upper bound in \eqref{main-b} holds for any subsolution $u$, whereas the lower bound in \eqref{main-b}  holds for any nontrivial supersolution $u$.

If $n \leq p < \infty$, then there are no nontrivial solutions to \eqref{inhom}.  
\end{Thm}

The proof of Theorem \ref{thm:main2} is based on a series of lemmas. 

\begin{Lem}\label{lemma-up}   Let $1<p<\infty$, 
$0<\alpha< \frac{n}{p}$, and $0<q<p-1$. Let $\sigma \in M^+(\R^n)$. Suppose $u$ is a subsolution to \eqref{q-eq}. Then 
\begin{equation}\label{main-bc}
u(x)\le \phi (x), \qquad x \in \R^n,  
\end{equation} 
provided $\W (u^q d \sigma) (x)<\infty$. In paticular, \eqref{main-bc} holds 
 $d\sigma$-a.e. 
\end{Lem}

\begin{proof} Setting $d \nu= u^q d \sigma$, we see that 
$u(x) \le \W \nu (x)<\infty$, and consequently $\W \nu (x) \le \W[(\W\nu)^q d \sigma](x)$.
Clearly, 
$$ 
\phi_\nu (x) := \W \nu (x) \, \left (\frac{\W[(\W\nu)^q d \sigma](x)}{\W \nu (x)}\right)^{\frac{p-1}{p-1-q}}\ge \W \nu (x).
$$
 Hence, 
$$
u(x)\le \phi_\nu (x), \quad x \in \R^n, 
$$ 
which yields  immediately \eqref{main-bc}. 
\end{proof}

\begin{Lem}\label{lemma-phi}   
Let $1<p<\infty$, 
$0<\alpha< \frac{n}{p}$, and $0<q<p-1$. Let $\nu, \sigma \in M^+(\R^n)$. Then there exists a positive constant $C$  which depends only on $p$, $q$, $\alpha$, and $n$ 
such that 
\begin{equation}\label{hom-in}
\begin{aligned}
\W[(\W\nu)^q d \sigma] & (x)  \le C \, (\W\nu(x))^{\frac{q}{p-1}} \\ \times & \left[\W\sigma(x) + (\K \sigma(x))^{\frac{p-1-q}{p-1}}  \right], \quad x \in \R^n.
\end{aligned}
\end{equation}
\end{Lem}

\begin{proof}  Without loss of generality we may assume that $\nu\not=0$ and $\W\nu(x)<\infty$. For $x \in \R^n$, we have  
\begin{align}\label{nu-sigma}
\W[(\W\nu)^q d \sigma](x) = \int_0^\infty \left[ \frac{\int_{B(x,t)} (\W\nu(y))^q d \sigma(y)}{t^{n-\alpha p}}\right]^{\frac{1}{p-1}}\frac{dt}{t}.
\end{align}
For $y\in B(x,t)$, we have that $B(y,r)\subset B(x, 2t)$ if $0<r\le t$, and $B(y,r)\subset B(x, 2r)$ if 
$r>t$. Consequently, for $y\in B(x,t)$,
\begin{align*}
& \W\nu(y) = \int_0^t \left[ \frac{\nu(B(y,r))}{r^{n-\alpha p}}\right]^{\frac{1}{p-1}}\frac{dr}{r} 
 +  \int_t^\infty \left[ \frac{\nu(B(y,r))}{r^{n-\alpha p}}\right]^{\frac{1}{p-1}}\frac{dr}{r} \\
&\le \int_0^t \left[ \frac{\nu(B(y,r)\cap B(x, 2t))}{r^{n-\alpha p}}\right]^{\frac{1}{p-1}}\frac{dr}{r} 
+ \int_t^\infty \left[ \frac{\nu(B(x,2r))}{r^{n-\alpha p}}\right]^{\frac{1}{p-1}}\frac{dr}{r}\\
& \le  \W\nu_{B(x, 2t)} (y) + c \, \W\nu(x),  
\end{align*}
were $c=2^{\frac{n-\alpha p}{p-1}}$. Hence, 
\begin{align*}
\int_{B(x,t)} (\W\nu(y))^q d \sigma(y)& \le \int_{B(x,t)}   \left(\W\nu_{B(x, 2t)}\right)^q  d \sigma(y) 
\\ & + c^q \, \left( \W\nu(x)\right)^q \, \sigma (B(x,t)). 
\end{align*}
Notice that by \eqref{kap-local}, 
$$
\int_{B(x,t)}  \left(\W\nu_{B(x, 2t)}\right)^q  d \sigma(y) \le \kappa(B(x, t))^{q} \, 
\nu(B(x, 2t))^{\frac{q}{p-1}}.
$$
Combining the preceding estimates, we deduce 
\begin{align*}
\int_{B(x,t)} (\W\nu(y))^q d \sigma(y)& \le   \kappa(B(x, t))^{q}  \, \nu(B(x, 2t))^{\frac{q}{p-1}} \\ & + c^q \, \left( \W\nu(x)\right)^q \, \sigma (B(x,t)).
\end{align*}
It follows from \eqref{nu-sigma} and the preceding estimate, 
\begin{align*}
& \W[(\W\nu)^q d \sigma](x) \\ & \le c \, \int_0^\infty 
 \left[   \frac{ \kappa(B(x, t))^{q}  \, \nu(B(x, 2t))^{\frac{q}{p-1}}} {t^{n-\alpha p}} \right]^{\frac{1}{p-1}} 
\frac{dt}{t} 
\\& + c \,  \left(\W\nu(x)\right)^{\frac{q}{p-1}} \, \int_0^\infty \left[ \frac{\sigma (B(x,t))}{t^{n-\alpha p}}\right]^{\frac{1}{p-1}}\frac{dt}{t} \\ & = c \, (I + II),
\end{align*}
 where $c=c(p, q, n)$. 
 
 By H\"{o}lder's inequality with exponents  $\frac{p-1}{p-1-q}$ and $\frac{p-1}{q}$, we estimate 
 \begin{align*}
& I= \int_0^\infty 
 \left[   \frac{ \kappa(B(x, t))^{q}  \, \nu(B(x, 2t))^{\frac{q}{p-1}}} {t^{n-\alpha p}} \right]^{\frac{1}{p-1}} \frac{dt}{t}  \\ & \le \left (  \int_0^\infty   \left[   \frac{\nu(B(x, 2t))} {t^{n-\alpha p}} \right]^{\frac{1}{p-1}} \frac{dt}{t}  
 \right)^{\frac{q}{p-1}}
 \\& \times \left ( \int_0^\infty 
  \left[   \frac{ \kappa(B(x, t)^{\frac {q(p-1)}{p-1-q}}} {t^{n-\alpha p}} \right]^{\frac{1}{p-1}} \frac{dt}{t}  
 \right)^{\frac{p-1-q}{p-1}} \\ & \le 2^{\frac{q(n-\alpha p)}{(p-1)^2}} \,  \left(\W\nu(x)\right)^{\frac{q}{p-1}}\, \left(\K \sigma(x)\right)^{\frac{p-1-q}{p-1}}.
 \end{align*}
Clearly, 
  \begin{align*}
  II=  \left(\W\nu(x)\right)^{\frac{q}{p-1}} \, \int_0^\infty \left[ \frac{\sigma (B(x,t))}{t^{n-\alpha p}}\right]^{\frac{1}{p-1}}\frac{dt}{t} =  \left(\W\nu(x)\right)^{\frac{q}{p-1}} \,  \W\sigma(x). 
   \end{align*}
   We deduce 
   \begin{align*}
 & \W[(\W\nu)^q d \sigma](x) \le c (I+II) \\ & \le c \, \left(\W\nu(x)\right)^{\frac{q}{p-1}} \,  \left[  \W\sigma(x) + \left(\K \sigma(x)\right)^{\frac{p-1-q}{p-1}} 
 \right].
  \end{align*}
  This completes the proof of \eqref{hom-in}.
   \end{proof}

   \begin{Lem}\label{lemma-phi1}   
Let $1<p<\infty$, 
$0<\alpha< \frac{n}{p}$, and $0<q<p-1$. Let $\sigma \in M^+(\R^n)$. Then there exist positive constants $C_1$, $C_2$  which depend only on $p$, $q$, $\alpha$ and $n$
such that 
\begin{equation}\label{hom-in1}
C_1 \, \phi(x) \le \left(\W\sigma(x)\right)^{\frac{p-1}{p-1-q}}  + \K \sigma(x)  \le C_2 \, \phi(x), 
 \end{equation}
 where the lower estimate holds at all $x \in \R^n$,  whereas the upper estimate holds 
 provided $\W\sigma(x)<\infty$ and  $\K \sigma(x)  < \infty$. 
\end{Lem}

\begin{Rem}\label{rem3.5} 
The assumptions $\W\sigma(x)<\infty$ and  $\K \sigma(x)  < \infty$
 in Lemma \ref{lemma-phi1}  can be replaced with $\W\sigma\not\equiv \infty$ and  $\K \sigma\not\equiv\infty$; then $\phi(x)<\infty$ $d \sigma$-a.e., and 
  \eqref{hom-in1} holds 
$d \sigma$-a.e. Moreover, the assumption $\W\sigma(x)<\infty$  in Lemma \ref{lemma-phi1} can be dropped altogether as shown below. 
\end{Rem} 

\begin{proof}[Proof of Lemma \ref{lemma-phi1}]  Let $\nu\in M^{+}(\R^n)$, $\nu\not=0$. 
Suppose $ \W\nu(x)<\infty$. 
Raising both sides of  \eqref{hom-in} to the power $\frac{p-1}{p-1-q}$ and multiplying 
by $ \W\nu(x)$, we obtain,  
\begin{align*} 
\phi_\nu(x):= & \W \nu (x) \, \left (\frac{\W[(\W\nu)^q d \sigma](x)}{\W \nu (x)}\right)^{\frac{p-1}{p-1-q}} \\& 
\le C^{\frac{p-1}{p-1-q}} \, \left[ \W\sigma(x)  + (\K \sigma(x))^{\frac{p-1-q}{p-1}} \right]^{\frac{p-1}{p-1-q}}.  
   \end{align*}
 The lower estimate in \eqref{hom-in1} follows immediately from the preceding inequality.

  To prove  the upper estimate in \eqref{hom-in1}, notice that,  
  since $\W\sigma(x)\not\equiv \infty$ and $\K \sigma(x)\not\equiv \infty$, it follows by 
 \cite[Theorem 4.8]{CV1} that there exists a (minimal) solution $u$ to \eqref{q-eq} such that 
  \begin{equation}\label{lower-upper}
  \begin{aligned}
  c _1\, & \left[  \left(\W\sigma(x)\right)^{\frac{p-1}{p-1-q}}  + \K \sigma(x) \right]\le  u(x) \\
  & \le  c_2 \, \left[  \left(\W\sigma(x)\right)^{\frac{p-1}{p-1-q}}  + \K \sigma(x) \right], \quad 
x \in \R^n. 
\end{aligned} 
      \end{equation}
  Here $c_1, c_2$ are positive constants which depend only on $p$, $q$, $\alpha$ and $n$, and  
 \eqref{lower-upper} holds at $x$ provided   
  $\W\sigma(x)< \infty$ and $\K \sigma(x)<\infty$.  Moreover, in this case $u(x)=\W(u^q d\sigma)(x)<\infty$. Thus, by Lemma \ref{lemma-up} and the lower bound in \eqref{lower-upper}, we have  
   $$
   c _1\, \left[  \left(\W\sigma(x)\right)^{\frac{p-1}{p-1-q}}  + \K \sigma(x) \right] \le u(x) \le \phi (x).  
   $$
  The proof of Lemma \ref{lemma-phi1} is complete. 
  \end{proof}
  
      \begin{proof}[Proof of Remark \ref{rem3.5}] If $\W\sigma(x)\not\equiv \infty$ and $\K \sigma\not\equiv \infty$, then as indicated in the above proof, there exists  a solution $u$ to \eqref{q-eq} such that $u= \W(u^q d \sigma)<\infty$ $d\sigma$-a.e., and 
      \eqref{lower-upper} holds $d\sigma$-a.e. In particular,  $\W\sigma(x)\not\equiv \infty$ and $\K \sigma\not\equiv \infty$ 
    $d\sigma$-a.e. Letting $d \nu = u^q d \sigma$, we 
     deduce 
     $u \le \phi_\nu \le \phi$ $d\sigma$-a.e., so that \eqref{hom-in1} 
     holds $d\sigma$-a.e. as well.

    Let us assume for a moment that $\W\sigma (x) < \infty$. 
 Then, letting $\nu=\sigma$ in the definition of $\phi_\nu$, we deduce by \cite[Lemma 3.5]{CV1} with $r=q$,  
 $$
 \W[(\W\sigma)^q d \sigma](x)\ge c \, \left (\W\sigma (x) \right)^{\frac{q}{p-1} +1}, 
 $$ 
 where $c$ is a positive constant which depends only on $p$, $q$, $\alpha$  and $n$. Hence, 
   \begin{align*}\label{nu-sigma}
  \phi(x) \ge \phi_\sigma(x) & = \W \sigma (x) \, \left (\frac{\W[(\W\sigma)^q d \sigma](x)}{\W \sigma (x)}\right)^{\frac{p-1}{p-1-q}} \\& \ge c \, (\W\sigma(x))^{\frac{p-1}{p-1-q}}. 
        \end{align*}

  Next, we observe that  
in the proof of the upper estimate in \eqref{hom-in1}, we may assume without loss of generality 
 that  $\W\sigma\not\equiv \infty$. 
 Otherwise, we may replace $\sigma$ with 
       $\sigma_{B(0, R)}$ for any $R>0$. Then  clearly $\W\sigma_{B(0, R)}\not\equiv \infty$,  
       and by the argument presented above (applied  to $\sigma_{B(0, R)}$ in place of $\sigma$), we see 
       that $\phi(x)\ge c \, (\W\sigma_{B(0, R)}(x))^{\frac{p-1}{p-1-q}}$. Letting $R\to \infty$ and using the 
       monotone convergence theorem, we see that the right-hand side tends to $\infty$ at every $x\in \R^n$, which forces 
    $\phi\equiv\infty$. 
 
    Finally, if  $\W\sigma(x)=\infty$, but  $\W\sigma\not\equiv \infty$, we may consider 
        $\W \sigma_k$, where $\sigma_k$ is the $p$-measure $-\Delta_p v_k =\sigma_k$, 
        so that $v_k \approx \W \sigma_k$, with 
        $v_k = \min (v, k)$ where $-\Delta_p v =\sigma$ and $v\approx \W \sigma$.  
        Notice that $\W \sigma_k(x)=k$. 
        
        Then, clearly,  
        $$
        \phi_{\sigma_k} (x) = k^{-\frac{q}{p-1-q}} \left( \W[(\W\sigma_k)^q d \sigma](x)\right)^{\frac{p-1}{p-1-q}}. 
        $$
        For $k>0$, we set $E_k=\{y:\, \W \sigma (y)\ge k \}$, so that $\W \sigma_{E_k} (y)=\W \sigma (y)$ for $y \in E_k$. 
         We estimate 
        $$
        \W[(\W\sigma_k)^q d \sigma](x) \ge k^{\frac{q}{p-1}} \W \sigma_{E_k}(x). 
        $$
        Thus,
         $$
        \phi(x)\ge \phi_{\sigma_k} (x) \ge  \left (\W \sigma_{E_k}(x)\right)^{\frac{p-1}{p-1-q}}. 
        $$
        Letting $k \to 0$, we see by the monotone convergence theorem that 
        $$
        \phi(x)\ge \phi_{\sigma_k} (x) \ge  \left (\W \sigma(x)\right)^{\frac{p-1}{p-1-q}} = \infty. 
        $$
        In other words, the assumption $\W \sigma (x)<\infty$ in Lemma \ref{lemma-phi1} 
        is redundant, and actually follows from 
       the fact that $\phi(x)<\infty$. 
       
\end{proof}

    \begin{proof}[Proof of Theorem \ref{thm:main2}] The upper bound in \eqref{main-b} 
    for any subsolution $u$ follows from Lemma \ref{lemma-up}, whereas the lower bound 
    for any nontrivial supersolution $u$ is a consequence of 
    Lemma \ref{lemma-phi} and \eqref{lower-upper}.
   \end{proof}

   As a consequence of the preceding results, we obtain the following corollary. 
    \begin{Cor}
   Under the assumptions of Theorem \ref{thm:main2}, there exist positive constants $C_1$, $C_2$ which depend 
   only on $p$, $q$, $\alpha$ and $n$ such that 
    \begin{equation}\label{phi-k-w} 
    \begin{aligned}
    C_1 \, & \left[ \left(\W\sigma(x)\right)^{\frac{p-1}{p-1-q}}  + \K \sigma(x) \right] \le u(x) \\& \le 
    C_2 \, \left[ \left(\W\sigma(x)\right)^{\frac{p-1}{p-1-q}}  + \K \sigma(x) \right], 
    \qquad x \in \R^n,  
    \end{aligned}
       \end{equation}
       for any solution $u$ to \eqref{q-eq}. Moreover, the lower estimate holds   for any supersolution $u$ such that \eqref{q-sup} holds at $x \in \R^n$, whereas the upper estimate holds for any subsolution $u$ 
       such that \eqref{q-sub} holds at $x \in \R^n$,  and also $d \sigma$-a.e.
   \end{Cor} 
   
   \begin{Rem} 
The upper estimate for $u$ in \eqref{phi-k-w} was proved earlier in \cite{CV1} only for the nontrivial \textit{minimal}  
 solution to \eqref{q-eq}, together with the lower estimate for any supersolution. 
\end{Rem}

\section{Non-homogeneous equations}\label{sec4}

In this section, we deduce estimates for sub- and super-solutions to the equation 
\begin{equation}\label{q-inhom}
u = \W (u^q d\sigma) + \W \mu, \quad u \ge 0  \quad {\rm in}  \,\,  \R^n, 
\end{equation}
in the case $0<q<p-1$ which immediately yields the corresponding estimates to solutions to \eqref{inhom} via the Wolff potential estimates \eqref{k-m}. 
The case $\mu=0$ was considered in Sec. \ref{sec3}, so we assume here that $\mu \not=0$. In particluar, 
all solutions $u$ to \eqref{q-inhom} are nontrivial: $u \ge \W \mu>0$, and $u<\infty$ $d \sigma$-a.e. (or q.e.) in $\R^n$.

\begin{Thm}\label{thm:main3}
Let $1<p<\infty$, 
$0<\alpha< \frac{n}{p}$, and $0<q<p-1$. Let $\sigma, \mu  \in M^{+}(\R^n)$. Then there exist positive constants $C_1, C_2$  which depend only  on $p$, $q$, $\alpha$ and $n$ such that any nonnegative solution $u$ to \eqref{q-inhom} satisfies the estimates 
\begin{equation}\label{main-c} 
\begin{aligned}
 C_1 \, & \left[  (\W \sigma(x))^{\frac{p-1}{p-1-q}} + \K \sigma(x)  + \W \mu (x)\right] \le u(x) \\ 
& \le  C_2 \, \left[  (\W \sigma(x))^{\frac{p-1}{p-1-q}} + \K \sigma(x) + \W \mu (x)\right], \qquad x \in \R^n, 
\end{aligned}
\end{equation}
where the upper estimate holds at every $x$ where $u(x)<\infty$,  and consequently $d \sigma$-a.e. and q.e.

Moreover, the lower  estimate in \eqref{main-c} holds for every supersolution $u$ at every $x \in \R^n$ such that 
\begin{equation}\label{main-c-sup} 
\W (u^q d\sigma)(x) +  \W \mu (x) \le u(x) <\infty, 
\end{equation}
whereas the upper estimate holds for every subsolution $u$ at every $x\in \R^n$ such that 
\begin{equation}\label{main-c-sub} 
u(x) \le  \W (u^q d\sigma)(x) +  \W \mu (x)<\infty.   
\end{equation}
\end{Thm}

\begin{proof}  The case $\mu=0$ is considered in Sec. \ref{sec3}, so we may assume without 
loss of generality that $\mu\not=0$. Consequently, $u(x) \ge \W \mu(x)>0$ at every $x \in \R^n$. 
Clearly, any supersolution to \eqref{q-inhom} is also a supersolution to \eqref{q-eq}.
Hence, by Theorem \ref{thm:main2},  there exists a positive constant $c=c(p, q, \alpha, n)$ such that 
$u(x) \ge c   \left[  (\W \sigma(x))^{\frac{p-1}{p-1-q}} + \K \sigma(x)\right]$.  
These two lower estimates combined yield the lower bound in  \eqref{main-c} with $C_1=C_1(p, q, \alpha, n)>0$.

To prove the upper bound, for any subsolution $u$ to \eqref{main-c}, we fix $x \in \R^n$  such that 
$u(x)\le \W (u^q d \sigma)(x)  + \W \mu(x)<\infty$. Notice that if $u$ is a solution to 
\eqref{main-c}, then this is equivalent to $u(x)<\infty$. 

Letting $d \omega =  u^q d \sigma + d \mu$ and  $c_1=\max \left(1, 2^{\frac{p-2}{p-1}}\right)$, we 
obviously have  $u(x)\le c_1 \,  \W \omega(x)<\infty$ at $x$ and $d \sigma$-a.e. Letting $c_2=\max \left(1, 2^{\frac{1}{2-p}}\right)$, 
we estimate 
\begin{align*}
 \W \omega(x) & =  \W (u^q d \sigma + d \mu) (x) 
 \\ & \le \,  c_2 \, \W (u^q d \sigma)(x) + c_2 \, \W \mu(x)
 \\ & \le \, c_1^{q} \, c_2 \,  \W [ (\W \omega)^q d \sigma](x) +  c_2 \,  \W \mu(x).
 \end{align*}
By Lemma \ref{lemma-phi} with $\omega$ in place of $\nu$, we have 
$$
\W [ (\W \omega)^q d \sigma](x)\le C \, (\W \omega(x))^{\frac{q}{p-1}} 
\left[ \left(\W \sigma (x)\right)^{\frac{p-1}{p-1-q}} + \K \sigma (x) 
\right]^{\frac{p-1-q}{p-1}},
$$
where $C=C(p, q, \alpha, n)$ is a positive constant. 
Combining the preceding estimates we deduce 
\begin{align*}
 \W \omega(x) \le c_1^{q} \, c_2  \, &  C \, (\W \omega(x))^{\frac{q}{p-1}}   \left[ \left(\W \sigma (x)\right)^{\frac{p-1}{p-1-q}} + \K \sigma (x) 
\right]^{\frac{p-1-q}{p-1}} \\ & + c_2 \, \W \mu(x).
 \end{align*}
Using Young's inequality with exponents $\frac{p-1}{q}$ and $\frac{p-1}{p-1-q}$ in the first term on the right-hand side, we estimate 
\begin{align*}
 \W \omega(x)  \le \frac{1}{2} \W \omega(x) + C' \, \left[ \left( \W \sigma (x)\right)^{\frac{p-1}{p-1-q}} + \K \sigma (x)\right] + c_2 \, \W \mu(x),
 \end{align*}
  where $C'$ is a positive constant which depends only on $p$, $q$, $\alpha$ and $n$. 
Since $ \W \omega(x)<\infty$, we can 
move the first term on the right to the left-hand side, and obtain 
\begin{align*}
u(x) \le c_1 \,   \W \omega(x)  \le C_2 \, \left[ \left(\W \sigma (x)\right)^{\frac{p-1}{p-1-q}} + \K \sigma (x) +  \W \mu(x)\right], 
 \end{align*}
where $C_2$ is a positive constant which depends only on $p$, $q$, $\alpha$ and $n$. 
 This completes the proof of the upper estimate in \eqref{main-c}. 
\end{proof}

    \begin{proof}[Proof of Theorem \ref{thm:main1}] Let 
    $d \omega = u^q d \sigma + d\mu$, where  $u$ is a solution to \eqref{inhom}. Then 
    by \eqref{k-m}, 
    \begin{equation}\label{k-m-om}
K^{-1}\W_p \omega (x) \leq u(x) \leq K \W_p \omega (x),
\end{equation} 
where $K=K(n,p)$ is a positive constant. Hence, $u$ is a supersolution 
satisfying  $u\ge \W_p (u^q d \tilde{\sigma}) + \W_p \tilde{\mu}$, with  
$\tilde{\mu}=c_1 \mu$ and $\tilde{\sigma}=c_2 \sigma$, where $c_1, c_2$ 
depend only on $p$, $q$, and $K$. Hence the 
lower estimate \eqref{main-b1}  of Theorem \ref{thm:main1} follows from 
the lower estimate \eqref{main-c} of Theorem \ref{thm:main3} 
in the special case 
$\alpha=1$. Similarly, the upper estimate in \eqref{main-b1} is deduced from the 
upper estimates in \eqref{main-c} and  \eqref{k-m-om}.  

If a nontrivial solution $u$ to  \eqref{inhom} exists, then by the 
lower estimate \eqref{main-b1}  of Theorem \ref{thm:main1} it follows that 
$\W_p \mu \not\equiv \infty$, $\W_p \sigma \not\equiv \infty$, and 
$\K_p \sigma \not\equiv \infty$, which are equivalent to conditions 
\eqref{suff-cond}, \eqref{k-m-cond}, and \eqref{suffcond1}, respectively. 

Conversely, suppose that these three conditions hold. In the special case $\mu=0$,
a positive $p$-superharmonic solution $u\in L^q_{{\rm loc}}(\R^n)$ was constructed in \cite{CV1}*{Theorem 1.1}  
by iterations, $u=\lim_{j\to \infty} u_j$, where $u_j$ is an nondecreasing sequence of 
$p$-superharmonic functions such that 
\begin{equation}\label{iter}
-\Delta_p u_{j+1} = \sigma u_j^q +  \mu \quad \text{in} \, \, \R^n, \quad j=0, 1, 2, \ldots,
\end{equation}
with an appropriate choice of $u_0$, namely $u_0=c \, \left(\W_p \sigma\right)^{\frac{p-1}{p-1-q}}$, where $c=c(p,q,n)$ is a small constant.

If $\mu\not=0$, a similar iteration argument can be used with $u_0=0$  
 based on \cite{PV2}*{Lemma 3.7 and Lemma 3.9}, so that $u_j$ satisfying \eqref{iter} 
 is an nondecreasing sequence of 
$p$-superharmonic functions. This part of the construction works for any 
$q>0$ and $p>1$ (see the proof of Theorem 3.10 in \cite{PV2} for $q>p-1$). However, 
the way we control the growth of $u_j$ for $0<q<p-1$ is very different.

Since $u_j\le u_{j+1}$, it follows that $u_{j+1}$ is a subsolution, so that 
\begin{equation}\label{iter-j}
-\Delta_p u_{j+1} \le \sigma u_{j+1}^q +  \mu \quad \text{in} \, \, \R^n, \quad j=0, 1, 2, \ldots,
\end{equation}
By  \eqref{k-m}, we have 
\begin{equation}\label{iter-k}
\begin{aligned}
 u_{j+1} & \le K \,  \W_p (\sigma u_{j}^q +  \mu) \\ & \le K \, \max (1, 2^{\frac{2-p}{p-1}}) \, \left[\W_p (\sigma u_{j+1}^q) 
 + \W_p \mu\right]. 
 \end{aligned}
\end{equation}
After scaling  by letting $\tilde \mu=c^{p-1} \mu$ and $\tilde \sigma=c^{p-1} \sigma$, 
where the constant $c=K  \max (1, 2^{\frac{2-p}{p-1}})$, we see that $u_{j+1}$ is a subsolution for the corresponding integral equation 
\eqref{q-inhom}, i.e., 
\begin{equation}\label{iter-m}
u_{j+1}\le \W_p (\tilde \sigma u_{j+1}^q) + \W_p \tilde\mu, \quad j=0, 1, 2, \ldots . 
\end{equation} 
It follows by induction using Lemma \ref{lemma-phi} 
with $\nu=\tilde \mu$ and $\nu=\tilde \sigma u_j^q $ 
that the right-hand side 
of \eqref{iter-m} is finite  at every point $x\in\R^n$ where $\W_p \mu(x)<\infty$, $\W_p \sigma(x)<\infty$, 
  and $\K_p \sigma(x) <\infty$ (which is true $d\sigma$-a.e., as we demonstrate below). 

By Theorem \ref{thm:main3}   for subsolutions, $u_{j+1}$ has the upper bound 
\[
u_{j+1}(x) \le C \left[\left(\W_p \sigma(x)\right)^{\frac{p-1}{p-1-q}} + \K_p \sigma(x) +
 \W_p \mu(x)\right],  \quad x \in \R^n,
\]
 with $C$  that depends only on $p$, $q$, and $n$, where we switched back from $\tilde \mu$, $\tilde \sigma$
 to $\mu$, $\sigma$. 
 
 Thus, $u=\lim_{j\to \infty} u_j$ satisfies 
 \begin{equation}\label{iter-l}
u(x) \le C \left[\left(\W_p \sigma(x)\right)^{\frac{p-1}{p-1-q}} + \K_p \sigma(x) +
 \W_p \mu(x)\right],  \quad x \in \R^n. 
\end{equation} 
   Moreover, by \cite{CV1}*{Theorem 1.1}, the conditions $\W_p \sigma\not\equiv\infty$  
  and $\K_p \sigma\not\equiv\infty$ yield the existence of a positive solution 
  $v \in L^q_{{\rm loc}} (\R^n, \sigma)$ to the homogeneous equation 
  \[
  -\Delta_p v = \sigma v^q  \quad \text{in} \, \, \R^n, 
  \]
   so that $v$ satisfies the lower bound 
   \[
   v\ge c \,  \left[\left(\W_p \sigma\right)^{\frac{p-1}{p-1-q}} + \K_p \sigma\right], 
   \]
   where $c>0$ is a constant that  depends only on $p$, $q$, and $n$. 
   Hence, $\left(\W_p \sigma\right)^{\frac{p-1}{p-1-q}} \in L^q_{{\rm loc}} (\R^n, \sigma)$ and  $\K_p \sigma \in L^q_{{\rm loc}} (\R^n, \sigma)$. 
   
   To verify that  $u  \in L^q_{{\rm loc}} (\R^n, \sigma)$, in view of \eqref{iter-l}, 
   it remains to show that  $ \W_p \mu \in L^q_{{\rm loc}} (\R^n, \sigma)$. Let $B=B(0, R)$ and let   $\mu=\mu_{2B} + \mu_{(2B)^c}$. Then,  as in the proof of 
   Lemma \ref{lemma-phi}, we clearly have for all $x\in B$,
   \begin{equation*}
\begin{aligned}
   \W_p \mu_{(2B)^c}(x) & =\int_{0}^{\infty} \left[ \frac{\mu(B(x,t)\cap (2B)^c)}{t^{n-p}}\right]^{\frac{1}{p-1}}\, \frac{dt}{t}\\ &\le 
   \int_{R}^{\infty} \left[ \frac{\mu(B(0, 2t)}{t^{n-p}}\right]^{\frac{1}{p-1}}\, \frac{dt}{t}\\& = 2^{\frac{n-p}{p-1}} \, \int_{2R}^{\infty} \left[ \frac{\mu(B(0, t)}{t^{n- p}}\right]^{\frac{1}{p-1}}\, \frac{dt}{t}. 
    \end{aligned}
\end{equation*} 
   
Hence,  
      \begin{equation*}
\begin{aligned}
  & \int_B (\W_p \mu)^q d \sigma  \le c \,   \int_B (\W_p \mu_{2B})^q d \sigma 
  +c \,  \int_B (\W_p \mu_{(2B)^c})^q d \sigma 
   \\ &\le   c \,   \varkappa(B)^q \,  \mu(2B)^{\frac{q}{p-1}} + 
 c \, \sigma (B) \, \left(\int_{2R}^{\infty} \left[ \frac{\mu(B(0,t))}{t^{n-\alpha p}}\right]^{\frac{1}{p-1}}\, \frac{dt}{t}\right)^q, 
       \end{aligned}
\end{equation*} 
   where $c>0$ is a constant that  depends only on $p$, $q$, and $n$. The right-hand side 
   of the preceding estimate is obviously finite by \eqref{suffcond1} and   
      \eqref{suff-cond}. This proves   that $u\in L^q_{{\rm loc}} (\R^n, \sigma)$.  
   
   Passing to the limit as $j\to \infty$ in  \eqref{iter}, we deduce as    in \cite{PV1}, \cite{PV2} for  $q>p-1$  
   that $u$ 
   is a positive $p$-superharmonic solution to \eqref{inhom}. 
 \end{proof}
   
   \begin{Cor}
   The results involving pointwise estimates,  as well as 
   necessary and sufficient conditions for 
   $u \in W^{1,p}_{{\rm loc}} (\R^n)$, $u\in L^r(\R^n)$,  $u \in  L^r_{{\rm loc}}(\R^n)$, etc., obtained in \cite{CV1}, \cite{CV2}, \cite{SV}, \cite{V}   for 
   minimal solutions $u$, 
   actually hold for   all solutions. 
   \end{Cor}

\begin{Rem}

1. In the case $p=2$, Theorem \ref{thm:main3}  yields bilateral pointwise estimates of solutions  to the fractional Laplace equation \eqref{eq:frac-laplacian}. The case of homogeneous equations $\mu=0$ was considered earlier  in \cite{CV1}, where the upper estimate was 
proved only for the minimal solution $u$.

2. As was mentioned in the Introduction, Theorem  \ref{thm:main1} is valid for general quasilinear $\mathcal{A}$-Laplace operators $\text{div}\mathcal{A}(x, \nabla u)$ in place of $\Delta_p$ under the standard 
 structural assumptions on  $\mathcal{A}$ which ensure that 
the Kilpel\"ainen--Mal\'y estimates \eqref{k-m} hold (see \cite{KiMa}, \cite{MZ}). 
The proofs in this setup are identical to those given above, with the 
same nonlinear potentials $\W_p$ and $\K_{p, q}$. The constants 
in our pointwise estimates \eqref{main-b1} then depend on the structural constants of $\mathcal{A}$. Analogous results also hold for $k$-Hessian equations in the case $0<q<k$ (see \cite{CV1} for $\mu=0$, and \cite{PV2} for $q>k$).

3. Complete analogues of our results for \eqref{inhom} hold for  
the non-homogeneous problem    
\begin{equation*}
\begin{cases}
-\Delta_{p} u  = \sigma u^{q} + \mu, \quad u\ge 0  \quad \text{in} \;\; \R^n, \\
\liminf\limits_{x\rightarrow \infty}u(x) = c,  
\end{cases}
\end{equation*}
where $c$ is a positive constant. One only needs to add $c$ to both sides 
of \eqref{main-b1}.

\end{Rem}


\end{document}